\documentclass{amsart}
\usepackage{amsmath, amssymb, amsthm, xcolor}
\usepackage{tikz}
\usepackage{caption}
\usepackage{subcaption}
\newtheorem{theorem}{Theorem}
\newtheorem{lemma}{Lemma}
\newtheorem{definition}{Definition}

\newtheorem{proposition}{Proposition}
\newtheorem{corollary}{Corollary}
\newtheorem{remark}{Remark}
\newtheorem{example}{Example}

\newcommand{\ts}{\mathsf{T}}

\begin{document}
\title[Optimal Ricci curvature MCMC methods]{Optimal Ricci curvature Markov chain Monte Carlo methods on finite states}
\author[Li]{Wuchen Li}
\email{wuchen@mailbox.sc.edu}
\address{Department of Mathematics, University of South Carolina, Columbia, SC 29208.}
\author[Lu]{Linyuan Lu}
\email{lu@math.sc.edu}
\address{Department of Mathematics, University of South Carolina, Columbia, SC 29208.}

\thanks{W Li is supported by AFOSR MURI FA9550-18-1-0502, and AFOSR YIP award 2023. Both W. Li and L. Lu are supported by NSF RTG: 2038080.}
\maketitle
\begin{abstract}
We construct a new Markov chain Monte Carlo method on finite states with optimal choices of acceptance-rejection ratio functions. We prove that the constructed continuous time Markov jumping process has a global in-time convergence rate in $L^1$ distance. The convergence rate is no less than one-half and is independent of the target distribution. For example, our method recovers the Metropolis–Hastings (MH) algorithm on a two-point state. And it forms a new algorithm for sampling general target distributions. Numerical examples are presented to demonstrate the effectiveness of the proposed algorithm. 
\end{abstract}

\section{Introduction}
Markov chain Monte Carlo (MCMC) methods \cite{MRRTT} are essential computational algorithms in scientific computing, statistics, and Bayesian inverse problems with applications in machine learning \cite{WKH, RS}. The MCMC method generates random samples from a target distribution, either in large dimensional sampling space or with intractable formulations. A typical MCMC method is the Metropolis–Hastings (MH) algorithm. It constructs an acceptance-rejection type Markov chain process, following which one generates the samples from target distributions. General MCMC methods have been widely studied in \cite{RR,RS}. 

The convergence analysis of the MCMC algorithm is a critical problem \cite{RS}. Recently, it has been known that the reversible Markov process forms a gradient flow in optimal transport-type metric spaces; see continuous states in \cite{am2006, BE, vil2008}, and discrete states in \cite{chow2012, maas2011gradient,M}, also named Onsager gradient flows in statistical physics \cite{GLL,ON}. The Hessian operators of relative entropies (divergence functions/free energies) in optimal transport-type metric spaces provide a convex analysis framework in establishing the convergence rates of MCMC methods. For example, on a finite state space, the convergence rate of the Markov jumping process follows from the smallest eigenvalue of Hessian matrices \cite{Maas2012,M1}. This smallest eigenvalue is called entropic Ricci curvature lower bound on a finite state Markov process. We also name them the smallest eigenvalue of mean field information Hessian matrices \cite{LiLuGraph}. 

From now on, we focus on constructing a continuous time finite-state Markov jumping process. A natural ``inverse problem'' arises. {\em Can we apply the convergence analysis in optimal transport-type metric spaces to construct a Markov jumping process for sampling a target distribution with optimal (largest) convergence rate? What are  finite-state MCMC algorithms with the optimal Ricci curvature lower bound?} 

In this paper, we design a finite state Markov jumping process by solving the optimal Ricci curvature lower bound problem. We construct a particular $Q$-matrix (generator) of the Markov jumping process. We then provide the exponential convergence analysis for the constructed continuous-time MCMC method in $\phi$-divergences, including the Kullback–Leibler divergence as an important example. We also show that the global-in-time convergence rate for any target distribution can be at least one-half in $L^1$ distance. We also use numerical experiments to verify that the proposed algorithm converges faster than the Metropolis–Hastings algorithm.

 The main result is sketched as follows. Given a finite state $I=\{1,\cdots, n\}$, suppose that there exists a target probability distribution $\pi=(\pi_i)_{i=1}^n\in \mathbb{R}_+^n$, $n\in\mathbb{N}$, with $\sum_{i=1}^n\pi_i=1$, we construct a $Q$-matrix, which is a generator of a continuous-time Markov jumping process:
\begin{equation}\label{Q}
 Q_{ij}= \left\{\begin{aligned}
 & \frac{\pi_j}{1-\min_{k\in I}\pi_k},\hspace{1cm}\textrm{if $j\neq i$}; \\
 & -\frac{1-\pi_i}{1-\min_{k\in I}\pi_k},\hspace{0.5cm}\textrm{if $j=i$}.
 \end{aligned}\right.
\end{equation}
The following convergence result for $Q$-matrix \eqref{Q} induced Markov jumping process holds.  
\begin{theorem}[Informal]
Denote $p(t)=(p_i(t))_{i=1}^n\in\mathbb{R}^n_+$, $t\geq 0$, as the probability distribution of Markov jumping process from $Q$-matrix \eqref{Q}. In other words, $p(t)$ satisfies the Kolomogrov forward equation:
\begin{equation*}
  \frac{dp_i(t)}{dt}=\sum_{j=1}^n \Big[Q_{ji}p_j(t)-Q_{ij}p_i(t)\Big],   
\end{equation*}
with an initial distribution $p(0)\in\mathbb{R}_+^n$, $\sum_{i=1}^np_i(0)=1$. 
Then 
\begin{equation*}
  \sum_{i=1}^n|p_i(t)-\pi_i|\leq C e^{-\kappa t}, 
\end{equation*}
where $C=\sqrt{2\sum_{i=1}^n p_i(0)\log\frac{p_i(0)}{\pi_i}}>0$. And $\kappa>0$ is a constant satisfying 
\begin{equation*}
    \kappa\geq \frac{1}{1-\min_{k\in I}\pi_k}\cdot \min_{i,j\in I}\left(1-\frac{1}{2}(\sqrt{\pi_i}-\sqrt{\pi_j})^2
    \right)\geq \frac{1}{2}. 
    \end{equation*}
\end{theorem}
We remark that if $n=2$, the $Q$-matrix \eqref{Q} is exactly the generator in Metropolis–Hastings algorithm on a two-point state. In other words, $Q_{12}=\min \Big\{1, \frac{\pi_2}{\pi_1}\Big\}$. 
When $n\geq 3$, the $Q$-matrix \eqref{Q} is different from the generator in Metropolis–Hastings algorithm; see Example \ref{three} in section \ref{sec5}. In literature, the convergence analysis of MCMC algorithms has been studied in \cite{WZ}. Compared to previous works, we develop a new MCMC sampling generator $Q$-matrix \eqref{Q}, which solves an inverse problem in optimizing convergence rates of finite state Markov processes. And the convergence rate is no less than one-half in $L_1$ distance, independent of the choices of any target distribution $\pi$. 
Our work is a natural step in transport information geometric convex analysis \cite{ LiHess, LiG, LiG1, LiLuGraph}. It is to design and compute fast convergence rate guaranteed MCMC algorithms. 

This paper is organized as follows. In section \ref{sec2}, we review some facts on the convergence analysis of reversible jumping Markov processes. The convergence rate is derived by the ``Ricci curvature lower bound'' based on the smallest eigenvalue of the Hessian matrix of Lyapunov functionals in terms of $\phi$-divergences (also named relative entropies or free energies). We also present the optimization problems of Ricci curvature lower bounds, which are to design the reversible Markov jumping process for sampling a target distribution. Using this optimality condition, we construct a Markov jumping process with the generator $Q$-matrix \eqref{Q}. Section \ref{sec3} presents the main result. We prove the global-in-time exponential convergence result for the constructed Markov jumping process with an analytical ``optimal convergence rate''. Proofs are given in section \ref{sec4}. We also present several examples of $Q$-matrices in two-point and three-point states in section \ref{sec5}. Numerical experiments in $250$, $500$, $1000$, and $2000$ states verify that the proposed algorithm converges faster than the Metropolis–Hastings algorithm.

\section{Optimal Ricci Curvature problems for finite state MCMC methods}\label{sec2}
In this section, we present the motivations of this paper. We first briefly review reversible Markov jumping processes and their constructions for continuous-time MCMC methods. We next present a constant, which determines convergence rates of Markov jumping processes. The constant is derived from the smallest eigenvalue (Ricci curvature lower bounds) of Hessian matrices of $\phi$-divergences in optimal transport-type metric spaces. The above two steps are ``forward problems'' to formulate an MCMC method and derive its convergence rate for a given target distribution. We last present an inverse problem. This is a class of optimization problems for designing Markov jumping processes with the largest convergence rates. By solving a ``local'' convergence rate problem, we derive the $Q$-matrix \eqref{Q}. 

\subsection{Reversible MCMC methods and symmetric weighted graphs}
In this subsection, we design an MCMC method using a symmetric weighted graph. It is a reversible Markov jumping process whose stationary distribution is a given target distribution. Suppose that there is a target probability distribution function satisfying
 \begin{equation*}
    \pi=(\pi_i)_{i=1}^n\in\mathbb{R}_+^n,\qquad \sum_{i=1}^n\pi_i=1.   
 \end{equation*}
We construct a Markov jumping process with a generator $Q$-matrix, whose stationary distribution satisfies $\pi$. 
\begin{definition}
Consider a symmetric weighted graph with self-loops $(I, \omega, E)$, where $I=\{1,\cdots, n\}$ is a vertex set, $\omega=(\omega_{ij})_{1\leq i,j\leq n}\in\mathbb{R}_+^{n\times n}$ is a symmetric weight matrix satisfying 
 \begin{equation*}
    \omega_{ij}=\omega_{ji}\geq 0, \quad \omega_{ii}\geq 0,  
 \end{equation*}
and $E=\{(i,j)\colon \omega_{ij}>0\}$ is an edge set.
Assume that 
\begin{equation*}
   \sum_{j=1}^n\omega_{ij}=\pi_i, \quad \textrm{for any $i\in I$}. 
\end{equation*}
Define a matrix $Q\in\mathbb{R}^{n\times n}$, such that
\begin{equation}\label{GQ}
Q_{ij}=\left\{\begin{aligned}
&\frac{\omega_{ij}}{\pi_i}, \hspace{2cm} \textrm{for $j\neq i$;}  \\
&-\sum_{k=1, k\neq i}^n \frac{\omega_{ik}}{\pi_i},\hspace{0.5cm} \textrm{for $j=i$}.
   \end{aligned}\right.
\end{equation}
\end{definition}
The $Q$-matrix \eqref{GQ} is a generator of a continuous-time reversible Markov chain on a finite state $\big\{1,2,\cdots, n\big\}$. In other words, the $Q$-matrix satisfies the row zero condition:
\begin{equation*}
Q_{ij}\geq 0, \quad \textrm{for $j\neq i$},\quad Q_{ii}=-\sum_{j=1,j\neq i}^nQ_{ij}.
\end{equation*}
The continuous-time Markov chain is reversible since the detailed balance relation holds:
\begin{equation}\label{db}
Q_{ij}\pi_i = Q_{ji} \pi_j=\omega_{ij}=\omega_{ji}.
\end{equation} 
And the Kolmogorov forward equation of the Markov process for the law $p_i(t)\in\mathbb{R}_+$, $i\in I$, satisfies
\begin{equation}\label{master}
\begin{split}
\frac{d p_i(t)}{d t}=\sum_{j=1}^n \Big[Q_{ji} p_j(t) - Q_{ij} p_i(t)\Big]
=\sum_{j=1}^n \omega_{ij}\Big[\frac{p_j(t)}{\pi_j}- \frac{p_i(t)}{\pi_i}\Big]. 
\end{split}
\end{equation}
 We note that $\pi$ is a stationary point of equation \eqref{master}. In other words, for any $i\in I$, 
 \begin{equation*}
\frac{d\pi_i}{dt}=\sum_{j=1}^n \Big[Q_{ji} \pi_j - Q_{ij} \pi_i\Big]=\sum_{j=1}^n \omega_{ij}\Big[\frac{\pi_j}{\pi_j}- \frac{\pi_i}{\pi_i}\Big] =0.
 \end{equation*}
\subsection{Lyapunov methods and convergence rates} 
 We next apply the Lyapunov method to study the convergence behavior of equation \eqref{master}. Define the $\phi$--divergence (relative entropy/free energy) on the finite state probability space: \begin{equation*}
 \mathrm{D}_{\phi}(p\|\pi) := \sum_{i=1}^n \phi\left (\frac{p_i}{\pi_i}\right) \pi_i, 
  \end{equation*} 
  where $\phi\in C^{2}(\mathbb{R}^1_+;\mathbb{R})$ is a convex function with $\phi(1)=0$ and $\phi'(1)=0$. Using the $\phi$-divergence as a Lyapunov function, equation \eqref{master} forms a gradient flow, known as the Onsager gradient flow \cite{ON}:
 \begin{equation}\label{ma1}
\begin{split}
  \frac{d p_i}{dt} =&\sum_{j=1}^n \omega_{ij}\Big[\frac{p_j}{\pi_j}- \frac{p_i}{\pi_i}\Big]\\
=&\sum_{j=1}^n \omega_{ij}\frac{\frac{p_j}{\pi_j}- \frac{p_i}{\pi_i}}{\phi'\left(\frac{p_j}{\pi_j}\right)- \phi'\left(\frac{p_i}{\pi_i}\right)}\Big[\phi'\left(\frac{p_j}{\pi_j}\right)- \phi'\left(\frac{p_i}{\pi_i}\right)\Big]\\ 
  =&\sum_{j=1}^n \theta_{ij}(\omega, p) \Big[\phi'\left(\frac{p_j}{\pi_j}\right)- \phi'\left(\frac{p_i}{\pi_i}\right)\Big]\\
   =&\sum_{j=1}^n \theta_{ij}(\omega, p) (\partial_{p_j}-\partial_{p_i})\mathrm{D}_\phi(p\|\pi),  
 \end{split}
 \end{equation}
 where we use the fact that $\partial_{p_i}\mathrm{D}_\phi(p\|\pi)=\phi'(\frac{p_i}{\pi_i})$ with 
\begin{equation*}
    \theta_{ij}(\omega, p):=\omega_{ij}\theta\left(\frac{p_i}{\pi_i}, \frac{p_j}{\pi_j}\right)\geq 0, \quad \textrm{and}\quad \theta(x, y) := \frac{x-y}{\phi'(x)- \phi'(y)}.
    \end{equation*}
Along the dynamics \eqref{master}, the $\phi$-divergence decays as follows:  
\begin{equation*}
\begin{split}
 \frac{d}{d t} \mathrm{D}_{\phi}(p(t)\|\pi) =& \sum_{i=1}^n\phi'\left(\frac{p_i(t)}{\pi_i}\right)\cdot \frac{dp_i(t)}{dt}\\
 =& \sum_{i=1}^n\phi'\left(\frac{p_i(t)}{\pi_i}\right)\sum_{j=1}^n \theta_{ij}(\omega, p(t)) \left[\phi'\left(\frac{p_j(t)}{\pi_j}\right)-\phi'\left(\frac{p_i(t)}{\pi_i}\right)\right]\\
 =& - \frac12 
 \sum_{i,j=1}^n \theta_{ij}(\omega, p(t)) \left[\phi'\left(\frac{p_j(t)}{\pi_j}\right)-\phi'\left(\frac{p_i(t)}{\pi_i}\right)\right]^2 \leq 0. 
\end{split}
 \end{equation*}

We then present the Hessian matrix of $\phi$-divergences along with dynamics \eqref{master}, using which we derive the convergence rate for equation \eqref{master}. The methods of computing Hessian matrices of Lyapunov functionals are known as Gamma calculus on graphs or mean-field information Hessian matrices; see \cite{LiLuGraph}. And the convergence rate is often named the ``generalized Ricci curvature lower bound'' or geodesic convexity in optimal transport-type metric spaces; see \cite{maas2011gradient, M1}. 
\begin{definition}
For any $f\in\mathbb{R}^n$, define a Gamma one operator: 
\begin{equation*}
    \Gamma_1(\omega, p,f,f)=\frac{1}{2}\sum_{i,j=1}^n (f_i-f_j)^2\theta_{ij}(\omega, p).
\end{equation*}
Define a Gamma two operator: 
\begin{equation*}
    \Gamma_2(\omega, p,f,f)=  \frac{1}{2} \sum_{i,j=1}^n (f_i-f_j)^2 a_{ij}(\omega, p), 
\end{equation*}
where 
\begin{equation*}
\begin{split}
a_{ij}(\omega, p):=&\frac{1}{2}\sum_{k=1}^n \Big[  
\frac{\partial \theta_{ij}}{\partial p_i} \eta_{ki}
+ \frac{\partial \eta_{ij}}{\partial p_i} \theta_{ki} 
+ \frac{\partial \eta_{jk}}{\partial p_j} \theta_{ij}
-\frac{\partial \eta_{ki}}{\partial p_k} \theta_{jk}\\
&\qquad-\frac{\partial \theta_{ij}}{\partial p_j} \eta_{jk}
- \frac{\partial \eta_{ij}}{\partial p_j} \theta_{jk} 
-\frac{\partial \eta_{ki}}{\partial p_i} \theta_{ij}
+\frac{\partial \eta_{jk}}{\partial p_k} \theta_{ki}\Big], 
\end{split}
\end{equation*}
with \begin{equation*}
    \eta_{ij}(\omega, p):=\theta_{ij}(\omega, p)(\partial_{p_j}-\partial_{p_i})\mathrm{D}_\phi(p\|\pi)=\omega_{ij}\left(\frac{p_j}{\pi_j}-\frac{p_i}{\pi_i}\right). 
\end{equation*}
\end{definition}

\begin{definition}[Convergence rate/Ricci curvature]
Define a largest possible scalar $\kappa(\omega, p)\in \mathbb{R}$, such that  
   \begin{equation*}
       \Gamma_2(\omega, p,f,f)\geq \kappa(\omega, p)\Gamma_1(\omega, p,f,f),
   \end{equation*}
   for any $f\in\mathbb{R}^n$. The following estimation holds: 
\begin{equation*}
 \kappa(\omega, p)\geq  \min_{(i,j)\in E}\frac{a_{ij}(\omega, p)}{\theta_{ij}(\omega, p)}.  
\end{equation*}
\end{definition}
 The following convergence result of $\phi$-divergences for dynamics \eqref{master} holds. 
\begin{corollary}[Convergence analysis \cite{LiLuGraph}]\label{CA}
Suppose that there exists a positive scalar $\kappa>0$, such that $\kappa(\omega, p)>\kappa>0$ for any $p\in\mathbb{R}_+^n$ with $\sum_{i=1}^np_i=1$, and $p(t)$ satisfies equation \eqref{master}, then 
\begin{equation*}
\mathrm{D}_{\phi}(p(t)\|\pi)\leq e^{-2\kappa t}\mathrm{D}_{\phi}(p(0)\|\pi),     
\end{equation*}
for any initial condition $p(0)\in\mathbb{R}_+^n$ with $\sum_{i=1}^np_i(0)=1$. 
\end{corollary}
\begin{proof}
From Lemma 5 in \cite{LiLuGraph}, we have 
\begin{equation*}
\frac{d}{dt}\mathrm{D}_\phi(p(t)\|\pi)=-\Gamma_1(\omega, p(t), f(t),f(t))|_{f(t)=\phi'(\frac{p(t)}{\pi})}, 
\end{equation*}
and 
\begin{equation*}
    \frac{d^2}{dt^2}\mathrm{D}_\phi(p(t)\|\pi)=2\Gamma_2(\omega, p(t), f(t),f(t))|_{f(t)=\phi'(\frac{p(t)}{\pi})}. 
\end{equation*}
From the definition of constant $\kappa$, we have
\begin{equation*}
\frac{d^2}{dt^2}\mathrm{D}_\phi(p(t)\|\pi)\geq -2\kappa\frac{d}{dt}\mathrm{D}_\phi(p(t)\|\pi).  
\end{equation*}
Integrating in a time domain $[t, \infty)$ with $t>0$ and using the fact that $\phi(1)=0$, $\phi'(1)=0$, we have
\begin{equation*}
\frac{d}{dt}\mathrm{D}_\phi(p(t)\|\pi)\leq -2\kappa \mathrm{D}_\phi(p(t)\|\pi).    
\end{equation*}
Following Grownwall's inequality, we prove the exponential convergence result for dynamics \eqref{master}. 
\end{proof}
\subsection{Optimal Ricci curvature problems}
We are ready to present the optimal Ricci curvature problem. We design a weighted matrix function $\omega\in\mathbb{R}^{n\times n}_+$, which maximizes the convergence rate $\kappa$ of equation \eqref{master}. 

We propose the following optimization problem, which maximizes the global-in-time convergence rate of dynamics \eqref{master}. 
\begin{definition}[Optimal global convergence rate problem]
Consider the following minimax problem: 
\begin{equation}\label{minRic}
    \max_{\omega\in \mathbb{R}_+^{n\times n}}\min_{p\in \mathbb{R}^n_+}~~{\kappa}(\omega, p),
\end{equation}
where the maximization is over all possible weight matrix $\omega\in\mathbb{R}_+^{n\times n}$ and the minimization is over all discrete probability function $p\in\mathbb{R}_+^n$, such that 
\begin{equation*}
\begin{aligned}
    &\sum_{i=1}^np_i=1,\quad p_i\geq 0, \quad i\in I,\\
   &\sum_{j=1}^n\omega_{ij}=\pi_i, \quad \omega_{ij}=\omega_{ji}\geq 0, \quad \omega_{ii}\geq 0, \quad i,j\in I. 
   \end{aligned}
\end{equation*}
\end{definition}
Solving the minimax problem \eqref{minRic} provides an ``optimal'' weight matrix function in which the probability density of MCMC methods converges to the target distribution at a desirable maximal global-in-time rate. In general, deriving an analytical optimality condition from the minimax problem \eqref{minRic} is not a simple task.

From now on, we propose an alternative approach. We obtain a simple variational problem, which only maximizes the local convergence rate of dynamics \eqref{master} near the stationary distribution. In other words, we maximizes the convergence rate $\frac{a_{ij}(\omega, p)}{\theta_{ij}(\omega, p)}$ at the stationary distribution $p=\pi$. 
\begin{definition}[Local optimal convergence rate problem]
Consider the following maximization problem:
\begin{equation}\label{relax}
    \max_{\omega\in \mathbb{R}_+^{n\times n}}\min_{(i,j)\in E}~~\frac{a_{ij}(\omega, \pi)}{\theta_{ij}(\omega, \pi)},
\end{equation}
where the maximization is over all possible weight matrix $\omega\in\mathbb{R}_+^{n\times n}$, such that
\begin{equation*}
\begin{aligned}
   &\sum_{j=1}^n\omega_{ij}=\pi_i, \quad \omega_{ij}=\omega_{ji}\geq 0, \quad \omega_{ii}\geq 0, \quad i,j\in I. 
   \end{aligned}
\end{equation*}
\end{definition}
{We next derive a class of weight functions in solving the critical point of variation problem \eqref{relax}. We present it below.} 
\begin{definition}
Denote $\omega^*\in\mathbb{R}_+^{n\times n}$, satisfies the following conditions:
There exists a constant $c>0$, such that 
\begin{equation*}
  c=\frac{1}{1-\min_{k\in I}\pi_k},
\end{equation*}
and 
\begin{equation*}
   \omega^*_{ij}=\left\{
   \begin{aligned}
   &c\pi_i\pi_j, \hspace{2.5cm} \textrm{for $i\neq j$;}  \\
& (1-c)\pi_i+c\pi_i^2, \hspace{1cm} \textrm{for $i=j$.}
\end{aligned}\right.
\end{equation*}
In this case, the Q-matrix \eqref{GQ} satisfies 
\begin{equation*}
Q_{ij}=\left\{\begin{aligned}
&c\pi_j, \hspace{2cm} \textrm{for $j\neq i$;}  \\
&-c(1-\pi_i),\hspace{0.7cm} \textrm{for $j=i$}.
   \end{aligned}\right.
\end{equation*}
This is exactly the $Q$-matrix \eqref{Q}. 
\end{definition}
\begin{proof}[Derivation of the weight matrix:]
We first prove the following claim. 

\noindent\textbf{Claim:}
The following identity holds:
\begin{equation*}
   a_{ij}(\omega, \pi)=\omega_{ij}\Big[\frac{\sum_{k\neq i}\omega_{ik}}{\pi_i}+ \frac{\sum_{k\neq j}\omega_{jk}}{\pi_j}\Big]-\sum_{k\neq i, k\neq j}\frac{\omega_{ik}\omega_{jk}}{\pi_k}. 
\end{equation*}
\begin{proof}[Proof of Claim \ref{sec2}]
We recall that $\Gamma_2(p,f,f)$ can be written as a quadratic form ${\frac{1}{2}}\sum_{i,j=1}^n a_{ij}(\omega, p)(f_i-f_j)^2$. We note that $\frac{\partial \eta_{ij}}{\partial p_i}=\frac{\omega_{ij}}{\pi_i}$ and $\frac{\partial \eta_{ij}}{\partial p_j}=-\frac{\omega_{ij}}{\pi_j}$. Thus 
\begin{align*}
a_{ij}(\omega, p)=&\frac{1}{2}\sum_{k=1}^n \Big(  
\frac{\partial \theta_{ij}}{\partial p_i} \eta_{ki}
+ \frac{\partial \eta_{ij}}{\partial p_i} \theta_{ki} 
+ \frac{\partial \eta_{jk}}{\partial p_j} \theta_{ij}
-\frac{\partial \eta_{ki}}{\partial p_k} \theta_{jk}\\
&\hspace{1cm}-\frac{\partial \theta_{ij}}{\partial p_j} \eta_{jk}
- \frac{\partial \eta_{ij}}{\partial p_j} \theta_{jk} 
-\frac{\partial \eta_{ki}}{\partial p_i} \theta_{ij}
+\frac{\partial \eta_{jk}}{\partial p_k} \theta_{ki}\Big)\\
=&-\frac{1}{2} \Big( \frac{\partial \theta_{ij}}{\partial p_i}
-\frac{\partial \theta_{ij}}{\partial p_j}\Big)
\eta_{ij} +  \Big( \frac{\partial \eta_{ij}}{\partial p_i}
-\frac{\partial \eta_{ij}}{\partial p_j}\Big)
\theta_{ij}\\
&+\frac{1}{2}\sum_{k\not = i,j}
\Big(
\frac{\partial \eta_{ij}}{\partial p_i} \theta_{ki} 
+ \frac{\partial \eta_{jk}}{\partial p_j} \theta_{ij}
-\frac{\partial \eta_{ki}}{\partial p_k} \theta_{jk}\\
&\hspace{1.5cm}
- \frac{\partial \eta_{ij}}{\partial p_j} \theta_{jk} 
-\frac{\partial \eta_{ki}}{\partial p_i} \theta_{ij}
+\frac{\partial \eta_{jk}}{\partial p_k} \theta_{ki}\Big).
\end{align*}
Note that when $p=\pi$, we have 
\begin{equation*}
  \frac{\partial\theta_{ij}(\omega ,p)}{\partial p_i}\eta_{ki}(\omega ,p)|_{p=\pi}=0,\quad \theta_{ij}(\omega, \pi)=\omega_{ij}, \quad \textrm{for any $i,j,k\in I$}.
\end{equation*}
Hence we obtain $a_{ij}(\omega, \pi)$ below:
\begin{align*}
a_{ij}(\omega, \pi)&=\omega_{ij}^2 \Big(\frac{1}{\pi_i}+\frac{1}{\pi_j}\Big)
 +\sum_{k\not = i,j}
\Big(\frac{\omega_{ij}\omega_{ik}}{\pi_i}
+ \frac{\omega_{ij}\omega_{jk}}{\pi_j}
-\frac{\omega_{ik}\omega_{jk}}{\pi_k}\Big).
\end{align*}
This finishes the proof. 
\end{proof}
From the Claim, we define 
\begin{equation*}
 F_{ij}(\omega):=\frac{ a_{ij}(\omega, \pi)}{\theta_{ij}(\omega, \pi)}= \frac{\sum_{k\neq i}\omega_{ik}}{\pi_i}+ \frac{\sum_{k\neq j}\omega_{jk}}{\pi_j}-\sum_{k\neq i, k\neq j}\frac{\omega_{ik}\omega_{jk}}{\pi_k\omega_{ij}}.
\end{equation*}
For $i,j,k\in I$, with $i\neq j$, $i\neq k$, we let
\begin{equation*}
  \frac{\partial}{\partial \omega_{ik}} F_{ij}(\omega)=\frac{1}{\pi_i}-\frac{\omega_{jk}}{\pi_k\omega_{ij}}=0. 
\end{equation*}
This implies 
\begin{equation}\label{eq}
   \frac{\pi_k}{\pi_i}=\frac{\omega_{jk}}{\omega_{ij}}, \quad \textrm{for any $i\neq j$, $i\neq k$, $i,j,k\in I$.}
\end{equation}
One solution of equation \eqref{eq} satisfies
\begin{equation*}
  \omega^*_{ij}=\omega_{ij}=c\pi_i\pi_j, \quad \textrm{for $i\neq j$}, 
\end{equation*}
where $c$ is a constant. And the constant $c$ is the solution of the following optimization problem:
\begin{equation*}
\max_{c\in \mathbb{R}_+}~~c \quad \textrm{s.t.}\quad
  c\geq 0, \quad (1-c)\pi_i+ c\pi_i^2\geq 0,\quad \textrm{for $i\in I$}.  
\end{equation*}
We have 
\begin{equation*}
  c\leq \frac{1}{1-\pi_i}, \quad \textrm{for any $i\in I$}.   
\end{equation*}
Thus, the optimal constant satisfies $c=\min_{k\in I}\{\frac{1}{1-\pi_k}\}$, which finishes the derivation. 
\end{proof}

\begin{remark}
We remark that both variational problems \eqref{minRic} and \eqref{relax} have closed-form critical points on a two-point space. One can show that it is to find optimal choices of weight matrices:
\begin{equation*}
   \max_{\omega_{12}}~~\Big\{\omega_{12} \colon\quad \omega_{11}+\omega_{12}=\pi_1,\quad \omega_{21}+\omega_{22}=\pi_2, \quad \omega_{ij}\geq 0, \quad i,j\in \{1,2\}\Big\}. 
\end{equation*}
Hence the optimal weight function satisfies $\omega^*_{12}=\min\{\pi_1, \pi_2\}$. Thus, 
$Q_{12}=\min\{1, \frac{\pi_2}{\pi_1}\}$, which is the generator in the Metropolis-Hastings algorithm on a two-point state. In other words, the Metropolis-Hastings algorithm on a two-point state maximizes the Ricci curvature/convergence rate. However, this is not the case when $n\geq 3$. Later on, in section \ref{sec5}, we show that the $Q$-matrix \eqref{Q} for $n=3$ is different from the generator in the Metropolis-Hastings algorithm.   
\end{remark}
\section{Main results}\label{sec3}
In this section, we always consider the $Q$-matrix \eqref{Q} and its associated Kolmogorov forward equation \eqref{master}. We use a $\phi$-divergence as the Lyapunov function and study the exponential convergence result of dynamics \eqref{master}. 
\begin{theorem}\label{thm2}
Denote $p(t)=(p_i(t))_{i=1}^n\in\mathbb{R}^n_+$, $t\geq 0$ satisfying the Kolmogorov forward equation \eqref{master}. We assume that $\phi\in C^{2}(\mathbb{R}_+;\mathbb{R})$, $\phi(x)$ is convex w.r.t. $x$ with $\phi''(x)\geq 0$, and $\phi(1)=\phi'(1)=0$. Then $\phi$-divergence converges to zero exponentially fast in time:
\begin{equation*}
\mathrm{D}_\phi(p(t)\|\pi)\leq e^{-2\kappa t}\mathrm{D}_\phi(p(0)\|\pi),
\end{equation*}
for any initial condition $p(0)\in\mathbb{R}_+^n$ with $\sum_{i=1}^np_i(0)=1$. And the global-in-time convergence rate $\kappa>0$ is a constant, satisfying 
\begin{equation*}
    \kappa= \frac{1}{1-\min_{k\in I}\pi_k}\cdot \min_{i,j\in I}\left(1-\frac{1}{2}(\pi_i+\pi_j) +\frac{1}{2} 
    \xi_{\phi}(\pi_i,\pi_j)\right).
    \end{equation*}
Here $\xi_\phi(\pi_i,\pi_j)$ is a positive two-variable function defined as
$$\xi_\phi(\pi_i,\pi_j)
=\inf_{(p_i,p_j)\in [0,1]^2}~\left(\phi''\left(\frac{p_i}{\pi_i}\right)\pi_j+\phi''\left(\frac{p_j}{\pi_j}\right)\pi_i\right)\cdot\frac{\frac{p_i}{\pi_i}-\frac{p_j}{\pi_j}}{\phi'(\frac{p_i}{\pi_i})-\phi'(\frac{p_j}{\pi_j})}\geq 0.
$$
In particular, the global-in-time convergence rate $\kappa$ is bounded below by $\frac{1}{2}$: 
\begin{equation*}
 \kappa\geq \frac{1}{2}.    
\end{equation*}
\end{theorem}
From now on, we consider the $\phi$-divergence as the alpha-divergence, see \cite{CA}. I.e., 
\begin{equation}\label{alpha}
\phi(x)=\left\{\begin{aligned}
&\frac{x^\alpha-1-\alpha(x-1)}{\alpha(\alpha-1)}, \hspace{1cm}\alpha\neq 0, 1;\\
& 1-x+x\log x, \hspace{1.6cm} \alpha=0;\\
&x-1-\log x, \hspace{1.9cm} \alpha=1. 
\end{aligned}\right.
\end{equation}
In the above formula, $\log$ is the natural logarithm function. In this case, we provide a detailed analysis of the two-variable functions $\xi_\phi$ and the global-in-time convergence rate $\kappa$. 
\begin{proposition}\label{prop2}
Suppose $\phi$ is defined in \eqref{alpha}. If $\alpha\in [0,2]$, then 
\begin{equation*}
     \xi_\phi(\pi_i,\pi_j)\geq 2\sqrt{\pi_i \pi_j},
\end{equation*}
and the global-in-time convergence rate $\kappa$ satisfies 
\begin{equation*}
   \kappa\geq \frac{1}{1-\min_{k\in I}\pi_k}\cdot \min_{i,j\in I}\left(1-\frac{1}{2}(\sqrt{\pi_i}-\sqrt{\pi_j})^2\right). 
\end{equation*}
\end{proposition}

We then present several examples of $\phi$-divergences. Typical examples include chi-squared ($\chi^2$), reverse Kullback–Leibler (KL), and KL divergences. Using them, we show the global-in-time convergence result of dynamics \eqref{master} with an analytical convergence rate. 
\begin{example}\label{ex1}
Let $\alpha=2$ and $\phi(x)=\frac{x^2-x}{2}$. Then the $\phi$-divergence forms the $\chi^2$-divergence:  
\begin{equation*}   \mathrm{D}_\phi(p\|\pi)=\frac{1}{2}\sum_{i=1}^n\frac{p_i^2}{\pi_i}.  \end{equation*}
In this case, we have 
\begin{equation*}
  \xi_\phi(\pi_i,\pi_j)= \pi_i+\pi_j.  
\end{equation*}
From Theorem \ref{thm2}, along dynamics \eqref{master}, the $\chi^2$-divergence converges to zero with an exponential convergence rate:
\begin{equation*}
   \kappa= \frac{1}{1-\min_{k\in I}\pi_k}.
\end{equation*}
\end{example}
\begin{example}\label{ex2}
Let $\alpha=1$ and $\phi(x)=x-1-\log x$. Then the $\phi$-divergence forms the reverse Kullback–Leibler divergence:  
\begin{equation*}
   \mathrm{D}_\phi(p\|\pi)=\sum_{i=1}^n\pi_i\log \frac{\pi_i}{p_i}.  \end{equation*}
In this case, we have 
\begin{equation*}
  \xi_\phi(\pi_i,\pi_j)=\inf_{(p_i, p_j)\in [0,1]^2} \Big(\frac{p_j\pi_i}{p_j}+\frac{p_i\pi_j}{p_j}\Big)=2\sqrt{\pi_i\pi_j} . 
\end{equation*}
From Theorem \ref{thm2}, along dynamics \eqref{master}, the reverse KL divergence converges to zero with an exponential convergence rate:
\begin{equation*}
\kappa=\frac{1}{1-\min_{k\in I}\pi_k}\cdot \min_{i,j\in I}\left(1-\frac{1}{2}(\sqrt{\pi_i}-\sqrt{\pi_j})^2
    \right).
\end{equation*}
\end{example}

\begin{example}
Let $\alpha=0$ and $\phi(x)=1-x+x\log x$. The $\phi$-divergence forms the KL divergence:
\begin{equation*}
 \mathrm{D}_\phi(p\|\pi)=\mathrm{D}_{\mathrm{KL}}(p\|\pi):=\sum_{i=1}^np_i\log\frac{p_i}{\pi_i}.   
\end{equation*}
And function $\xi_{\mathrm{KL}}(\pi_i, \pi_j):=\xi_\phi(\pi_i, \pi_j)$ defines a particular symmetric two-variable function below. Define a function $\xi\colon \mathbb{R}^2_+ \to \mathbb{R}_+$ as
$$\xi(s,t):=\inf_{x\in (0,1)\cup(1,\infty)} \frac{(\frac{s}{x} +t)(x-1)}{\log x}.$$
\begin{lemma}\label{lemma1}
We have the following properties.
\begin{enumerate}
\item[(i)] $\xi(\pi_i, \pi_j)=\xi_{\mathrm{KL}}(\pi_i,\pi_j)$;
\item[(ii)] $2\sqrt{st}\leq \xi(s,t)\leq \frac{2(s-t)}{\log(s)-\log(t)}.$
\end{enumerate}
In addition, 
\begin{equation*}
 \xi_{\mathrm{KL}}(\pi_i,\pi_j)\geq 2\sqrt{\pi_i\pi_j}.    
\end{equation*}
\end{lemma}
 Following Lemma \ref{lemma1}, we show the exponential convergence result of dynamics \eqref{master}. We then derive an analytical convergence rate for the proposed weight function $\omega_{ij}$, which is no less than $\frac{1}{2}$.   
\begin{theorem}\label{thm3}
Denote $p(t)=(p_i(t))_{i=1}^n$, $t\geq 0$, satisfying the Kolmogorov forward equation \eqref{master}. 
Then,  
\begin{itemize}
\item[(i)] the KL divergence converges to zero exponentially in time:
\begin{equation*}
\mathrm{D}_{\mathrm{KL}}(p(t)\|\pi)\leq e^{-2\kappa t}\mathrm{D}_{\mathrm{KL}}(p(0)\|\pi);    
\end{equation*}
\item[(ii)]
the $L_1$ distance converges to zero exponentially in time:
\begin{equation*}  
\sum_{i=1}^n|p_i(t)-\pi_i|\leq \sqrt{2\mathrm{D}_{\mathrm{KL}}(p(0)\|\pi)}e^{-\kappa t},
\end{equation*}
\end{itemize}
for any initial condition $p(0)\in\mathbb{R}_+^n$ with $\sum_{i=1}^np_i(0)=1$. And the global-in-time convergence rate $\kappa>0$ is a constant, satisfying 
\begin{equation*}
    \kappa\geq \frac{1}{1-\min_{k\in I}\pi_k}\cdot \min_{i,j\in I}\left(1-\frac{1}{2}(\sqrt{\pi_i}-\sqrt{\pi_j})^2\right)\geq\frac{1}{2}.
    \end{equation*}

\end{theorem}

\end{example}

\subsection{Proof}\label{sec4}
We present all proofs in this subsection. 
\begin{proof}[Proof of Theorem \ref{thm2}]
Given a probability distribution $\pi=(\pi_1,\ldots, \pi_n)^{\ts}$, denote the $\phi$-divergence (energy function) as $E(p)=\mathrm{D}_\phi(p\|\pi)$. Since
\begin{equation*}
   \omega_{ij}=c\pi_i\pi_j,  \qquad c=\frac{1}{1-\min_{k\in I}\pi_k}, 
\end{equation*}
we have
\begin{equation*}
\begin{aligned}
\theta_{ij}&=c\pi_i\pi_j\frac{\frac{p_i}{\pi_i}-\frac{p_j}{\pi_j}}{\frac{\partial E(p)}{\partial p_i} -\frac{\partial E(p)}{\partial p_j}}\\
&=c\frac{\pi_jp_i-\pi_ip_j}{\phi'(\frac{p_i}{\pi_i})-\phi'(\frac{p_j}{\pi_j})}.
\end{aligned}
\end{equation*}
Hence
$$\eta_{ij}=c(\pi_jp_i-\pi_ip_j).$$
Then
$\Gamma_2(p,f,f)$ can be written as a quadratic form ${\frac{1}{2}}\sum_{i,j=1}^n a_{ij}(f_i-f_j)^2$, where 
\begin{align*}
a_{ij}=&\frac{1}{2}\sum_{k=1}^n \Big(  
\frac{\partial \theta_{ij}}{\partial p_i} \eta_{ki}
+ \frac{\partial \eta_{ij}}{\partial p_i} \theta_{ki} 
+ \frac{\partial \eta_{jk}}{\partial p_j} \theta_{ij}
-\frac{\partial \eta_{ki}}{\partial p_k} \theta_{jk}\\
&\qquad-\frac{\partial \theta_{ij}}{\partial p_j} \eta_{jk}
- \frac{\partial \eta_{ij}}{\partial p_j} \theta_{jk} 
-\frac{\partial \eta_{ki}}{\partial p_i} \theta_{ij}
+\frac{\partial \eta_{jk}}{\partial p_k} \theta_{ki}\Big)\\
=&-\frac{1}{2} \Big( \frac{\partial \theta_{ij}}{\partial p_i}
-\frac{\partial \theta_{ij}}{\partial p_j}\Big)
\eta_{ij} +  \Big( \frac{\partial \eta_{ij}}{\partial p_i}
-\frac{\partial \eta_{ij}}{\partial p_j}\Big)
\theta_{ij}\\
& +\frac{1}{2}\sum_{k\not = i,j}
\Big(
\frac{\partial \eta_{ij}}{\partial p_i} \theta_{ki} 
+ \frac{\partial \eta_{jk}}{\partial p_j} \theta_{ij}
-\frac{\partial \eta_{ki}}{\partial p_k} \theta_{jk}\\
&\hspace{1.5cm}
- \frac{\partial \eta_{ij}}{\partial p_j} \theta_{jk} 
-\frac{\partial \eta_{ki}}{\partial p_i} \theta_{ij}
+\frac{\partial \eta_{jk}}{\partial p_k} \theta_{ki}\Big)\\
=&-\frac{1}{2} \Big( \frac{\partial \theta_{ij}}{\partial p_i}
-\frac{\partial \theta_{ij}}{\partial p_j}\Big)
\eta_{ij} +  \Big( \frac{\partial \eta_{ij}}{\partial p_i}
-\frac{\partial \eta_{ij}}{\partial p_j}\Big)
\theta_{ij}\\
& +\frac{1}{2}\sum_{k\not = i,j}
\Big(
c\pi_j\theta_{ki} 
+ c\pi_k \theta_{ij}
-c\pi_i\theta_{jk}\\
&\hspace{1.8cm}+ c\pi_i \theta_{jk} 
+c\pi_k \theta_{ij}
-c\pi_j \theta_{ki}\Big)\\
=&-\frac{1}{2} \Big( \frac{\partial \theta_{ij}}{\partial p_i}
-\frac{\partial \theta_{ij}}{\partial p_j}\Big)
\eta_{ij} +  c(\pi_j+\pi_i)\theta_{ij}
 + c\sum_{k\not=i,j} \pi_k \theta_{ij}\\
=&-\frac{1}{2} \Big( \frac{\partial \theta_{ij}}{\partial p_i}
-\frac{\partial \theta_{ij}}{\partial p_j}\Big)
\eta_{ij}  + c\theta_{ij}\sum_{k=1}^n \pi_k\\
=& -\frac{c}{2} \Big( \frac{\partial \theta_{ij}}{\partial p_i} 
-\frac{\partial \theta_{ij}}{\partial p_j}\Big)(\pi_jp_i-\pi_i p_j)
+c \theta_{ij}.
\end{align*}
Thus, we have
$$\frac{a_{ij}}{\theta_{ij}} = 
c\left(1-\frac{1}{2}\left(\frac{\partial\log \theta_{ij}}{\partial p_i} 
-\frac{\partial\log \theta_{ij}}{\partial p_j}\right)(\pi_jp_i-\pi_i p_j)
\right).$$
Denote $\frac{\partial}{\partial p_i}-\frac{\partial}{\partial p_j}$ as  $\partial_{\vec{ij}}$.
Then $$\log \theta_{ij}=\log(c) + \log(\pi_j p_i - \pi_i p_j) - \log (\partial_{\vec{ij}}E),$$
and 
$$\partial_{\vec{ij}}\log \theta_{ij} =\frac{\pi_j+\pi_i}{\pi_j p_i - \pi_i p_j}- \frac{\partial^2_{\vec{ij}}E} {\partial_{\vec{ij}}E}.
$$
We have
\begin{equation}
\frac{a_{ij}}{\theta_{ij}} 
= c\left(1-\frac{\pi_i+\pi_j}{2}\right)
+\frac{1}{2}\theta_{ij} \partial^2_{\vec{ij}}E.
\end{equation}
If $E(p)$ is convex, then $\theta_{ij}>0$ and $\partial^2_{\vec{ij}}E\geq 0$. Thus 
\begin{equation}
\kappa\geq c\left(1-\frac{\pi_i+\pi_j}{2}\right ).
\end{equation}
Since $c\geq 1$ and $\frac{\pi_i+\pi_j}{2}\leq \frac{1}{2}$, we have  $\kappa\geq \frac{1}{2}$. Following Corollary \ref{CA}, we have 
\begin{equation*}
  \mathrm{D}_{\phi}(p(t)\|\pi)\leq e^{-2\kappa t}   \mathrm{D}_{\phi}(p(0)\|\pi). 
\end{equation*}
\end{proof}
\begin{proof}[Proof of Proposition \ref{prop2}]
Since $\phi$ is defined in equation \eqref{alpha}, then 
\begin{equation*}
\phi'(x)=\frac{x^{\alpha-1}}{\alpha-1},\qquad \phi''(x)=x^{\alpha-2}.    
\end{equation*}
Thus, the two variable function $\xi_\phi(\pi_i,\pi_j)$ satisfies 
$$\xi_\phi(\pi_i,\pi_j)
=\inf_{(p_i,p_j)\in [0,1]^2}~\left(\left(\frac{p_i}{\pi_i}\right)^{\alpha-2}\pi_j+\left(\frac{p_j}{\pi_j}\right)^{\alpha-2}\pi_i\right)\cdot\frac{\frac{p_i}{\pi_i}-\frac{p_j}{\pi_j}}{\frac{1}{\alpha-1}((\frac{p_i}{\pi_i})^{\alpha-1}-(\frac{p_j}{\pi_j})^{\alpha-1})}.
$$
To show $\xi_\phi(\pi_i, \pi_j)\geq 2\sqrt{\pi_i, \pi_j}$, we need to prove 
\begin{equation}\label{Ineq}
 u^{\alpha-2}\pi_j+v^{\alpha-2}\pi_i\geq \frac{2\sqrt{\pi_i\pi_j}}{\alpha-1}\frac{u^{\alpha-1}-v^{\alpha-1}}{u-v},   
\end{equation}
where $u=\frac{p_i}{\pi_i}$ and $v=\frac{p_j}{\pi_j}$. Using the Cauchy-Schwarz inequality on the L.H.S. of \eqref{Ineq}, we shall show the following inequality:
\begin{equation}\label{Ineq1}
  2\sqrt{\pi_i\pi_j}  (uv)^{\frac{\alpha-2}{2}}\geq   \frac{2\sqrt{\pi_i\pi_j}}{\alpha-1} \frac{u^{\alpha-1}-v^{\alpha-1}}{u-v}. 
\end{equation}
Denote $w=(\frac{u}{v})^{\frac{1}{2}}$. Dividing $2\sqrt{\pi_i\pi_j}(u v)^{\frac{\alpha-2}{2}}$ on both sides of inequality \eqref{Ineq1}, we have to show that 
\begin{equation*}
  1\geq \frac{1}{\alpha-1}\frac{w^{\alpha-1}-\frac{1}{w^{\alpha-1}}}{w-\frac{1}{w}}. 
\end{equation*}
Denote $w=e^{z}$, and $\beta=\alpha-1$. Since $\alpha\in[0,2]$, then $\beta\in [-1,1]$. Thus the above inequality \eqref{Ineq1} satisfies 
\begin{equation*}
    1\geq \frac{1}{\beta}\frac{e^{\beta z}-e^{-\beta z}}{e^z-e^{-z}}=\frac{1}{|\beta|}\frac{\sinh(|\beta|z)}{\sinh(z)}, \qquad \textrm{for $|\beta|\in [0,1]$}. 
\end{equation*}
Note that $\sinh(a)=\frac{e^{a}-e^{-a}}{2}$. We shall prove 
\begin{equation*}
  H(z):= |\beta|\sinh(z)-\sinh (|\beta| z)\geq 0,  \quad \textrm{for $|\beta|\in [0,1]$ and $z\geq 0$.}
\end{equation*}
 When $|\beta|\leq 1$, we have 
\begin{equation*}
    H'(z)=|\beta|\big(\cosh(z)-\cosh(|\beta|z)\big)\geq 0. 
\end{equation*}
Thus $H(z)$ is an increasing function with $H(0)=0$. Hence $H(z)\geq 0$. This finishes the proof. 
\end{proof}

\begin{proof}[Proof of Lemma \ref{lemma1}]
If $\phi(x)=x\log x+1-x$, then $\phi'(x)=\log x$ and $\phi''(x)=\frac{1}{x}$. Hence 
\begin{equation*}
\begin{split}
 \pi_i\pi_j\left(\frac{1}{p_i}+\frac{1}{p_j}\right)\frac{\frac{p_i}{\pi_i}-\frac{p_j}{\pi_j}}{
\log\frac{p_i}{\pi_i}-\log\frac{p_j}{\pi_j}}
=&\left(\pi_j\frac{\pi_i p_j}{p_i\pi_j}+\pi_i\right)\frac{\frac{p_i\pi_j}{\pi_i p_j}-1}{\log\frac{p_i\pi_j}{\pi_i p_j}}.
\end{split}
\end{equation*}
Denote $x=\frac{p_i\pi_j}{p_j\pi_i}$. Then we prove (i). 

We next prove (ii). Define the following function: 
\begin{equation*}
   F(x,s,t)=\left\{\begin{aligned}
   &\frac{(\frac{s}{x}+t)(x-1)}{\log x}, \hspace{1cm} \textrm{if $x\neq 1$};\\
   & s+t,\hspace{2.5cm} \textrm{if $x=1$.}
\end{aligned}\right.
\end{equation*}
Clearly, $\xi(s,t)=\inf_{x\in\mathbb{R}_+}F(x,s,t)$. Hence 
\begin{equation*}
   \xi(s,t)\leq F\left(\frac{s}{t}, s, t\right)=\frac{2(s-t)}{\log s-\log t}.  
\end{equation*}
We next prove $F(x,s,t)\geq 2\sqrt{st}$. Denote 
\begin{equation*}
   G(x)=\left(\frac{s}{x}+t\right)(x-1)-2\sqrt{st}\log x.  
\end{equation*}
We note that $G(x)$ is an increasing function when  $x\in\mathbb{R}_+$. Since 
\begin{equation*}
  G'(x)=\frac{s}{x^2}+t-2\frac{\sqrt{st}}{x}=\left(\frac{\sqrt{s}}{x}-\sqrt{t}\right)^2\geq 0.   
\end{equation*}
Since $G(1)=0$ and $G$ is an increasing function, we have 
\begin{equation*}
  \left\{\begin{aligned}
    &  G(x)\leq 0,\quad \textrm{if $x<1$;}\\
    &  G(x)> 0,\quad \textrm{if $x>1$.}\\
  \end{aligned}\right.  
\end{equation*}
This finishes the proof of (ii).

From (i) and (ii), we have 
\begin{equation*}
 \xi_{\mathrm{KL}}(\pi_i, \pi_j)=\xi(\pi_i, \pi_j)\geq 2\sqrt{\pi_i\pi_j}.    
\end{equation*}
\end{proof}
\begin{proof}[Proof of Theorem \ref{thm3}]
When $\phi(x)=x\log x$ and  $\mathrm{D}_\phi(p\|\pi)=\mathrm{D}_{\mathrm{KL}}(p||\pi)$. From Lemma \ref{lemma1}, we have
\begin{align*}
\frac{a_{ij}}{\theta_{ij}}&=c\left(1-\frac{1}{2}(\pi_i+\pi_j) + \frac{\pi_i\pi_j}{2}\left(\frac{1}{p_i}+\frac{1}{p_j}\right)\frac{\frac{p_i}{\pi_i}-\frac{p_j}{\pi_j}}{
\log\frac{p_i}{\pi_i}-\log\frac{p_j}{\pi_j}
}\right)\\
&\geq c\left(1-\frac{1}{2}(\pi_i+\pi_j)  
+\frac{1}{2}\xi_{\mathrm{KL}}(\pi_i,\pi_j)\right)\\
&\geq  c\left(1-\frac{1}{2}(\pi_i+\pi_j) 
+\sqrt{\pi_i\pi_j}\right)\\
&=c\left(1-\frac{1}{2}(\sqrt{\pi_i}-\sqrt{\pi_j})^2 \right).
\end{align*}
Hence 
\begin{equation*}
   \kappa\geq  c\left(1-\frac{1}{2}(\sqrt{\pi_i}-\sqrt{\pi_j})^2 \right). 
\end{equation*}
From Corollary \ref{CA}, we have 
\begin{equation*}
  \mathrm{D}_{\mathrm{KL}}(p(t)\|\pi)\leq e^{-2\kappa t}   \mathrm{D}_{\mathrm{KL}}(p(0)\|\pi). 
\end{equation*}
From Pinsker's inequality, we obtain 
\begin{equation*}
     \sum_{i=1}^n|p_i(t)-\pi_i| \leq \sqrt{2 \mathrm{D}_{\mathrm{KL}}(p(t)\|\pi)}\leq e^{-\kappa t}   \sqrt{2\mathrm{D}_{\mathrm{KL}}(p(0)\|\pi)}. 
\end{equation*}
This finishes the proof. 
\end{proof}

\begin{remark}
Our proof shows that the global convergence rate depends on the estimation of function $\xi_\phi$, which forms a class of two-variable functions; see an example in Lemma \ref{lemma1}. We only use a rough estimate with $\xi_{\mathrm{KL}}(s,t)\geq 2\sqrt{st}$. 
In future works, we shall study the properties and lower bounds of $\xi_\phi$ in general $\phi$-divergence functions.  
\end{remark}
\section{Examples}\label{sec5}
This section presents several analytical examples of $Q$-matrix \eqref{Q}. We also provide numerical examples to verify exponential convergence results of dynamics \eqref{master}.  
\subsection{Analytical example}
\begin{example}[Two point state]
Given $\pi=(\pi_1,\pi_2)^{\ts}\in\mathbb{R}_+^2$ with $\pi_1+\pi_2=1$, denote 
$c=\min \Big\{\frac{1}{1-\pi_1}, \frac{1}{1-\pi_2}\Big\}= \min \Big\{\frac{1}{\pi_2}, \frac{1}{\pi_1}\Big\}$. Then the $Q$-matrix \eqref{Q} satisfies 
\begin{equation*}
\begin{split}
   Q=&\begin{pmatrix}
       -c(1-\pi_1) & c\pi_2  \\   
         c\pi_1 & -c(1-\pi_2)
   \end{pmatrix}\\
=&\begin{pmatrix}
       -\min\{1,\frac{\pi_2}{\pi_1}\} & \min\{1,\frac{\pi_2}{\pi_1}\}  \\   
          \min\{1,\frac{\pi_1}{\pi_2}\} & - \min\{1,\frac{\pi_1}{\pi_2}\}
   \end{pmatrix},
\end{split}
\end{equation*}
which is exactly the $Q$-matrix in Metropolis–Hastings algorithms. 
\end{example}

\begin{example}[Three point state]\label{three}
Given $\pi=(\pi_1,\pi_2, \pi_3)^{\ts}\in\mathbb{R}_+^3$ with $\pi_1+\pi_2+\pi_3=1$, denote 
$c=\min \Big\{\frac{1}{1-\pi_1}, \frac{1}{1-\pi_2}, \frac{1}{1-\pi_3}\Big\}$. 
The $Q$-matrix \eqref{Q} satisfies 
\begin{equation*}
\begin{split}
   Q=&\begin{pmatrix}
       -c(1-\pi_1) & c\pi_2 & c\pi_3\\    
         c\pi_1 & -c(1-\pi_2)  & c\pi_3\\
        c\pi_1 & c\pi_2 & -c(1-\pi_3)
   \end{pmatrix}.
\end{split}
\end{equation*}
We also compare the proposed $Q$-matrix \eqref{Q} with the generator in the Metropolis-Hastings algorithm, denoted as $Q^{\mathrm{MH}}$. And the $Q^{\mathrm{MH}}$-matrix can be computed blow: 
\begin{equation*}
\footnotesize
   \frac{1}{2}\begin{pmatrix}
       -\min\{1, \frac{\pi_2}{\pi_1}\}- \min\{1, \frac{\pi_3}{\pi_1}\} & \min\{1, \frac{\pi_2}{\pi_1}\} & \min\{1, \frac{\pi_3}{\pi_1}\}\\    
         \min\{1, \frac{\pi_1}{\pi_2}\} & - \min\{1, \frac{\pi_1}{\pi_2}\}-\min\{1, \frac{\pi_3}{\pi_2}\} &    \min\{1, \frac{\pi_3}{\pi_2}\}\\
        \min\{1, \frac{\pi_1}{\pi_3}\} &  \min\{1, \frac{\pi_2}{\pi_3}\} & -\min\{1, \frac{\pi_1}{\pi_3}\}-\min\{1, \frac{\pi_2}{\pi_3}\}
   \end{pmatrix}. 
\end{equation*}
We remark that $Q\neq Q^{\mathrm{MH}}$.
\end{example}
\subsection{Numerical examples}
In this subsection, we also numerically verify the proposed exponential convergence results. We randomly choose $p^0(\omega_k)\neq \pi(\omega_k)\in \mathbb{R}_+^n$, with $\sum_{i=1}^np^0_i(\omega_k)=\sum_{i=1}^n\pi(\omega_k)=1$, where $\omega_k$ is a random sampling realization with $k=1,\cdots, K=100$.  
For each realization $\omega_k$, we compute a forward Euler time discretization of ODE \eqref{master}:
\begin{equation}\label{Euler}
  \frac{p^{N+1}_i(\omega_k)-p^N_i(\omega_k)}{\Delta t}=\sum_{j=1}^n\Big[Q_{ji}p_j^N(\omega_k)-Q_{ij}p_i^N(\omega_k)\Big],   
\end{equation}
where $p^0(\omega_k)\in\mathbb{R}_+^n$ is a random initial point and $\Delta t=0.01$ is a stepsize. Two choices of matrices are considered. One is the $Q$-matrix \eqref{Q}, the other is the $Q^{\mathrm{MH}}=(Q^{\mathrm{MH}}_{ij})_{1\leq i,j\leq n}$, where $Q^{\mathrm{MH}}_{ij}=\frac{1}{n-1}\min\{1,\frac{\pi_j}{\pi_i}\}$, if $j\neq i$, and $Q^{\mathrm{MH}}_{ii}=-\sum_{k=1,k\neq i}^nQ^{\mathrm{MH}}_{ik}$. We also run the forward Euler method \eqref{Euler} with matrix $Q$ and matrix $Q^{\mathrm{MH}}$ for $N=0,1,\cdots, \frac{T}{\Delta t}$, and $T=10$. After computing all time steps and realizations, we compute the sample average $L^1$ distance between $p^N$ and $\pi$:
\begin{equation*}
   \frac{1}{K}\sum_{k=1}^K\sum_{i=1}^n |p^N(\omega_k)-\pi(\omega_k)|.  
\end{equation*}
In Figure \ref{fig:four_graph}, we plot the convergence result of average $L^1$ distances between $p^N$ and $\pi$ in terms of the time variable. It shows that the dynamics \eqref{master} with the proposed $Q$-matrix \eqref{Q} converges faster than $Q^{\mathrm{MH}}$-matrix in Metropolis-Hastings algorithm. 

\begin{figure}
     \centering
     \begin{subfigure}[b]{0.47\textwidth}
         \centering
         \includegraphics[width=\textwidth]{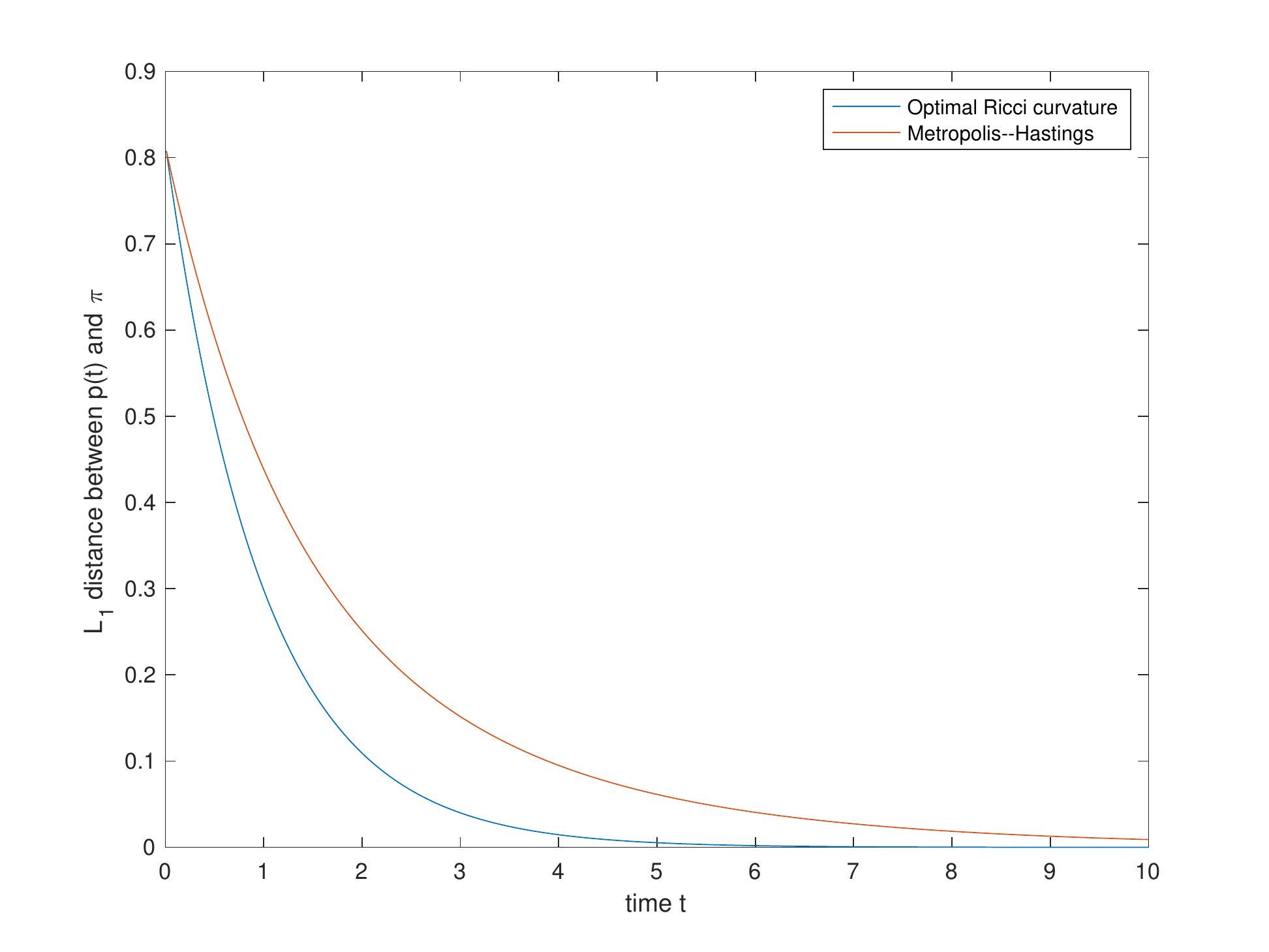}
         \caption{$n=250$.}
     \end{subfigure}
     \hfill
     \begin{subfigure}[b]{0.47\textwidth}
         \centering
         \includegraphics[width=\textwidth]{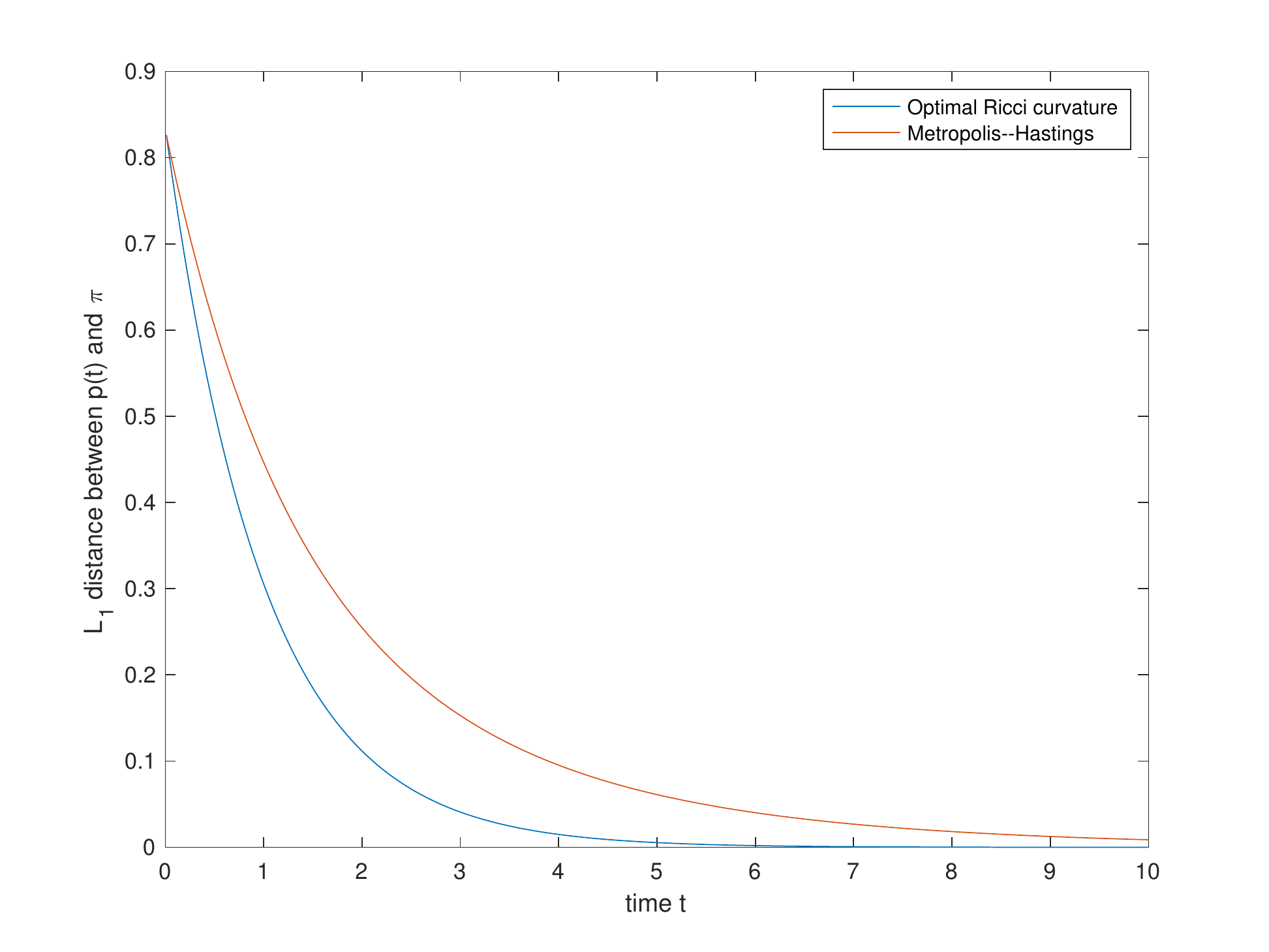}
         \caption{$n=500$.}
     \end{subfigure}
     \hfill
     \begin{subfigure}[b]{0.47\textwidth}
         \centering
         \includegraphics[width=\textwidth]{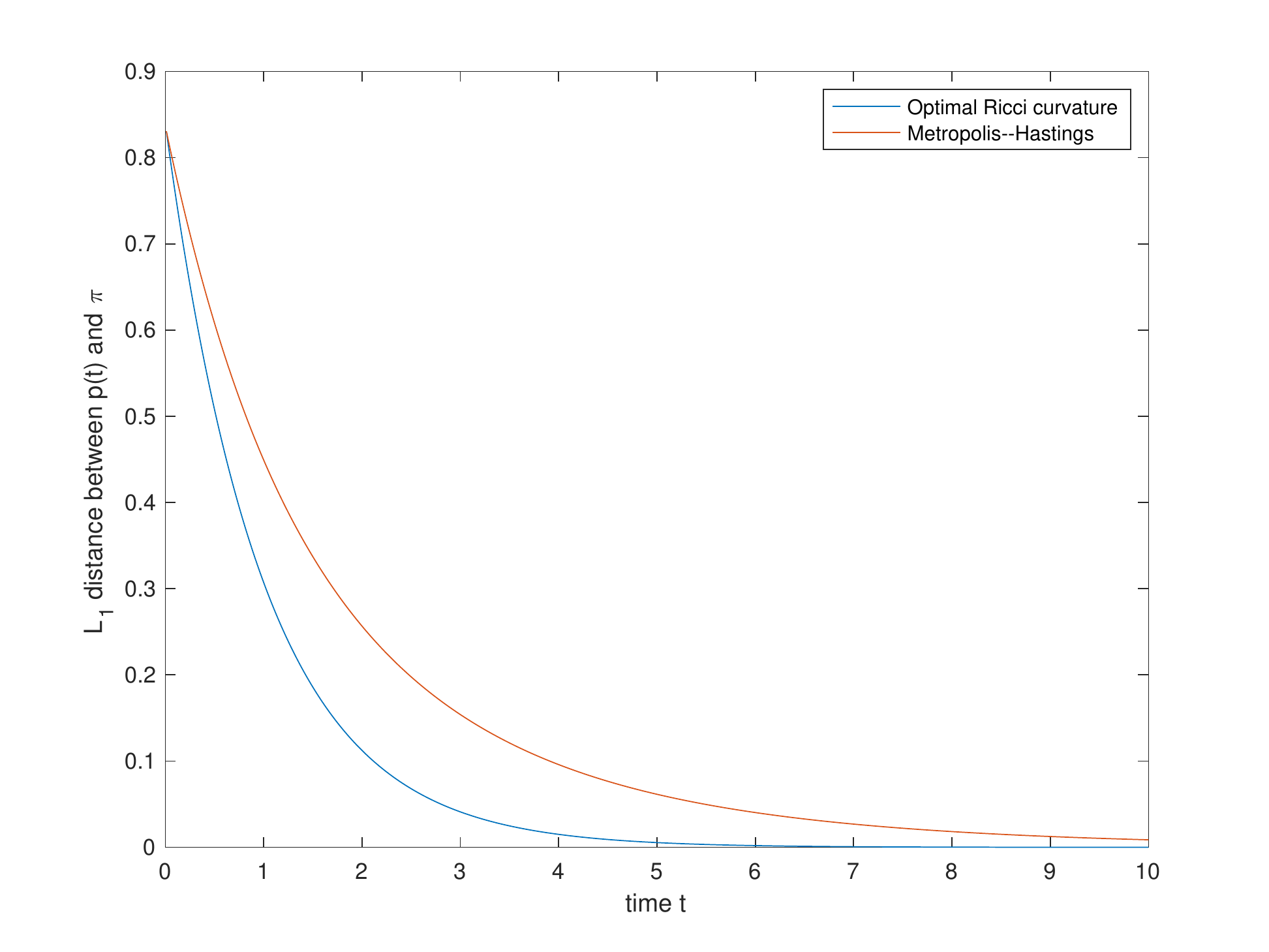}
         \caption{$n=1000$.}
     \end{subfigure}
          \hfill
     \begin{subfigure}[b]{0.47\textwidth}
         \centering
         \includegraphics[width=\textwidth]{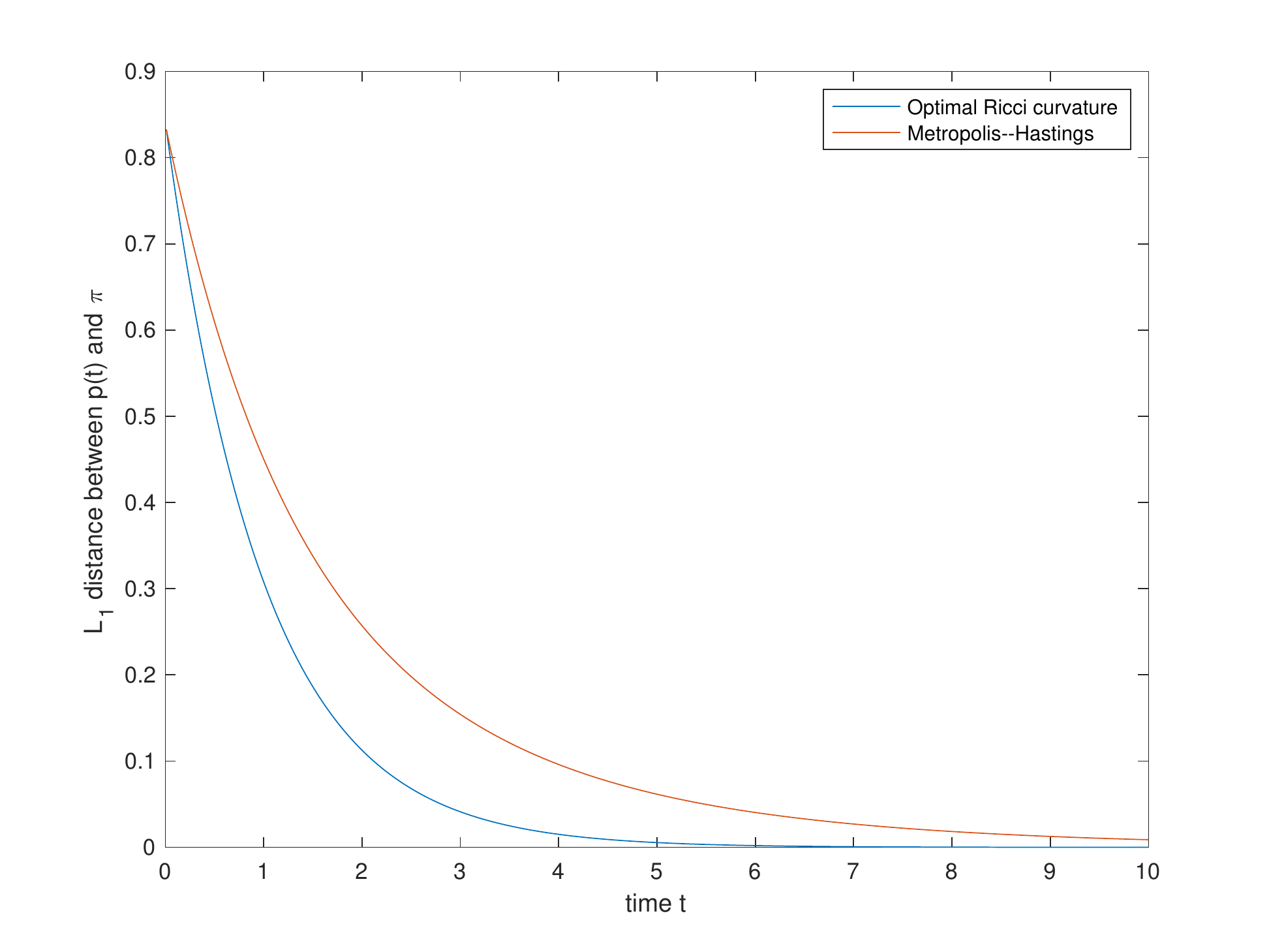}
         \caption{$n=2000$.}
     \end{subfigure}
        \caption{The above four plots demonstrate the convergence results of dynamics\eqref{master} for $n=250, 500, 1000, 2000$. The $x$-axis represents the time, and the $y$-axis represents the $L_1$ distance between $p(t)$ and $\pi$. 
         The blue curve represents the convergence behavior of the proposed $Q$ matrix \eqref{Q}, while the red curve demonstrates the convergence behavior of the Metropolis-Hastings algorithm. }
        \label{fig:four_graph}
\end{figure}

\section{Discussions}
We also observe some insights from the eigenvalues of the $Q$-matrix \eqref{Q}. We rewrite the $Q$-matrix \eqref{Q} as follows:
\begin{equation*}
   Q=c \Big(I-(1,\cdots,1) (\pi_1,\cdots, \pi_n)^{\ts}\Big),
\end{equation*}
where $I\in\mathbb{R}^{n\times n}$ is an identity matrix and  $(1,\cdots,1) (\pi_1,\cdots, \pi_n)^{\ts}$ is a rank-one matrix. Thus, the $ Q$ matrix \eqref{Q} is a rank-one modification of the identity matrix, whose most eigenvalues are $-1$. Moreover, as shown in this paper, we prove the convergence result of master equation \eqref{master} with $Q$-matrix \eqref{Q}, where the convergence rate is guaranteed to be independent of any target distribution $\pi$. 

We note that the current construction of the weight function is dense. It is essentially a complete graph. This brings challenges in MCMC computations with extremely large states. In future work, we shall design a sparse graph weight function with maximal Ricci curvature lower bound. We are also working on the convergence analysis for discrete-time Markov chains. This is also based on mean field information matrices calculations \cite{LiLuGraph}.

\end{document}